\documentclass[11pt]{amsart}    
\usepackage{mathrsfs, amsrefs}
\usepackage{epsfig}       
\usepackage{subfig}     
\usepackage{amssymb,amscd,bbm}     
\usepackage{amsthm,amsgen,amsmath,epic,psfrag}
\usepackage{rotating}
\usepackage{enumitem} 
\usepackage[dvips]{xypic}       
\usepackage{spectralsequences}
\usepackage{tikz-cd}

\xyoption{curve}                
\xyoption{rotate}               

\newcommand{\Ext}{{\rm Ext}}
\newcommand{\Hom}{{\rm Hom}}

\newcommand{\R}{{\mathbb R}}

\newcommand{\F}{{\mathbb F}}

\newcommand{\Z}{\mathbb{Z}}

\newcommand{\tr}{\odot}

\newcommand{\si}{{\mathcal S}}
\newcommand{\A}{\mathcal A}
\newcommand{\Sq}{{Sq}}

\newcommand{\ua}{\mathcal{U}\mathcal{A}}
\newcommand{\ind}{{\rm Ind}\;}

\newcommand{\cu}{\mathcal U}
\newcommand{\g}{\gamma}
\newcommand{\lm}{\lambda}
\newcommand{\bi}{\overline{\iota}}
\newcommand{\coker}{{\rm coker}}
\newcommand{\im}{{im}}
 
\theoremstyle{plain}
\newtheorem{theorem}{Theorem}[section]
\newtheorem{proposition}[theorem]{Proposition}

\newtheorem{corollary}[theorem]{Corollary}
\newtheorem{conjecture}[theorem]{Conjecture}

\theoremstyle{definition}
\newtheorem{definition}[theorem]{Definition}

\theoremstyle{remark}

\input{colordvi}

 \title{The Curtis-Wellington spectral sequence through cohomology}

\author{Dana Hunter}

  \thanks{This work was partially supported by NSF DMS-2039316.}
\begin{document}
\maketitle

\begin{abstract}
We study stable homotopy through unstable methods applied to its representing infinite loop space, as
pioneered by Curtis and Wellington.  Using cohomology instead of homology, we find a width filtration whose subquotients
are simple quotients of Dickson algebras.  We make initial calculations and determine towers in the resulting width spectral sequence.
We also make  calculations related to the image of $J$ and conjecture that it is captured exactly by the lowest filtration 
in the width spectral sequence.
\end{abstract}

\section{Introduction}

Our project, first taken up by Curtis \cite{Curtis} and Wellington  \cite{Wellington}, is to study stable homotopy groups at the prime two through the unstable Adams spectral sequence for their representing space.  We call this 
 the Curtis-Wellington spectral sequence (CWSS), which to our knowledge has not been studied above the zero line for almost 
 forty years.

\subsection{Main Results}
Recall that $Q_0 S^0$ is the zero component of the infinite loop space which represents stable homotopy.
Recall as well that the Dickson algebras are rings of invariants $D_n = \F_2[x_1, \ldots, x_n]^{Gl_n(F_2)}$, calculated by Dickson as polynomial on
generators in degree $2^n - 2^\ell$.  The Steenrod algebra action on the ambient polynomial algebras restrict
to the Dickson algebras, which provide a rich and still not fully understood collection of (unstable) modules
over the Steenrod algebra. Let $D_n^o$ be the quotient of $D_n$ by all perfect squares.

\begin{theorem}\label{trigradedSSIntro}
There is a width filtration (see Definition \ref{D:widthFiltration} below) of  the indecomposables of  $H^*(Q_0S^0, \F_2)$ 
whose subquotients are $D_n^o$.  
\end{theorem}

Applying a reduction of Bousfield (see Proposition \ref{abelianE2} below) we have the following.

\begin{corollary}
There is  
a width spectral sequence with 
	\[ E_1^{s, t;n } = \Ext_{\cu}^{s,t}( D_n^o, \F_2) \]
	and $d_r: \Ext^{s,t}( D_n^o) \to \Ext^{s+1,t}( D_{n+r}^o)$ which  converges to 
the $E_2$-term of the Curtis-Wellington spectral sequence.
	\end{corollary}

Our calculations rely on presentation of the cohomology of $Q_0 S^0$ due to Giusti-Salvatore-Sinha.
Through that presentation we can already manage by-hand calculations of the Curtis-Wellington
spectral sequence more readily than by previous techniques.  But filtration by (quotients of) Dickson algebras
is particularly amenable to computer calculation, for which we can thank Hood Chatham - see Appendix \ref{F:FullChart}.  

It is elementary to eliminate the possibility of many differentials in the width spectral sequence, so these
computer calculations
provide a good understanding of the $E_2$-term of the Curtis-Wellington spectral sequence  as well.
While, disappointingly, the $E_2$ term of the CWSS is much larger than even the classical Adams spectral sequence,
we have explored two regular phenomena.

In Section \ref{TowersSec} we make calculations to determine the locations of towers in the resulting width spectral sequence. 

\begin{theorem} \label{D1 towers intro}
		There are infinite towers in $\Ext^{s,t}(\Sigma^{-1}D_1^o, \F_2)$ in degrees $4a - 2$ for $a$ a positive integer. 
	\end{theorem}

	\begin{theorem}\label{Dn towers intro}
		Let $n$ be an integer greater than or equal to 2. There are towers in $\Ext^{s,t}(\Sigma^{-1}D_n^o, \F_2)$ corresponding to all integer solutions of
		\[ (2^{n-2}-2^{n-3})a_1 + \cdots + (2^{n-2}-1)a_{n-2} + (2^{n-1}-1)b_{n-1} = k \]
		where at least one of $a_1, \ldots, a_{n-2}$ are odd and $n \ge 3$. And also towers corresponding to all integer solutions of
		\[ (2^{n-1}-2^{n-2})b_1 + \cdots + (2^{n-1}-1)b_{n-1} + (2^{n}-1)c_{n} + (2^n-1) = k. \]
	\end{theorem}

These results agree with the locations of towers identified by  Wellington in the $E_2$ term of the CWSS,
  thus imply that there are no differentials between the towers internal to the width spectral sequence. \\

Finally, in Section \ref{JSec}
 we make calculations relating to the image of $J$, seeing preliminarily that it seems to be governed 
by the width filtration. We noticed that the unstable Adams $\Ext$ chart for $H^*(BO, \F_2)$ is a shifted version of the unstable Adams $\Ext$ chart for the first quotient of the width filtration. 

\begin{proposition}\label{shiftedD1 Intro}
$\Ext_{\cu}^{s,t}(\Sigma^{-1}  D_1^o) \cong \Ext_{\cu}^{s,{t+1}}(\Sigma^{-1} \ind H^*(BO))$
\end{proposition} 

While it seems this would be classical, we haven't found any treatment of this in the literature. This motivated us to study the image of $J$ map on cohomology, with preliminary calculations that support the following. 

\begin{conjecture}\label{imJ conjecture intro}
The algebraic map on $\Ext$ induced by the image of $J$ on cohomology is the same as the map on $\Ext$ induced by reduction to $D_1^o$. 
\end{conjecture}

We provide initial evidence for this conjecture, and also speculate that the map on unstable Adams spectral sequence
$E_2$ terms is algebraic.  Such results would invite further study of the compatibility of the width filtration
and the chromatic filtration.

\subsection{Background}

Recall that the stable homotopy groups of spheres are isomorphic to the unstable homotopy groups of  $Q_0 S^0$, the degree zero component of $\varinjlim \Omega^d S^d$. 
The unstable Adams spectral sequence for $Q_0 S^0$, which we choose to name the Curtis-Wellington 
spectral sequence after the only two people to study it globally, was introduced by Curtis in \cite{Curtis}.
 He outlined some first calculations, noticing that Adams filtration lowered and that both the Hopf and Kervaire classes were in filtration zero, leading to the well-known and still open conjecture that these are the only classes to survive on the zero line.  But Curtis made some fundamental errors, which Wellington corrected and then went on to establish  more global properties, including the classification of Bockstein towers (which in particular
precludes any upper vanishing lines).

In order to make these calculations, 
Curtis and Wellington applied a  deep connection between stable homotopy and symmetric groups.  
At this level of homology, this was noticed independently and simultaneously by Barratt and Priddy \cite{BP} and Quillen.   Briefly, one models the classifying space for the $n$th symmetric group as a colimit over $d$ of space of $n$ disks in $\R^d$. Then given a set of $n$ disks in $\R^d$, associate to it a collapse map from $S^d = \R^d \cup \infty$ to itself which sends the complement of the disks to the base point and each interior of a disk homeomorphically onto the $S^d \backslash \infty$.   These maps from the space of disks to $\Omega^d S^d$ can be assembled to a map from the colimit. The Barratt-Priddy-Quillen theorem (BPQ) tells us that a resulting map from the classifying space for the infinite symmetric groups to $Q_0 S^0$ is an isomorphism in homology.  


Wellington and Curtis  then applied the known structure of the homology of $Q_0 S^0$.  This homology is a free algebra under the product induced by the connect sum of maps, whose generators are the Kudo-Araki-Dyer-Lashof algebra acting on a single class \cite{CLM}. While algebraic topologists are comfortable with ``homological coalgebra,'' in this case one runs into difficulties because calculations of the homology coproduct as well as Steenrod coaction (Nishida relations) require regular applications of Adem relations.  Wellington had to filter carefully to make things at all tractable. Through non-explicit methods, Nakaoka had previously shown that the cohomology of the infinite symmetric group and thus $Q_0 S^0$ is polynomial, generated in combinatorially interesting degrees.  Finer control of that calculation, and in particular incorporation of Steenrod algebra action, motivated many authors to study cohomology of symmetric groups in more detail in the eighties and nineties \cite{AdemMaginnisMilgram, AdemMilgram, Feshbach}. Relatively recently,  Giusti, Salvatore and Sinha found a new Hopf ring presentation for the cohomology of symmetric groups, as algebras over the Steenrod algebra \cite{GSS}, yielding a ``skyline diagram'' presentation for the cohomology of  $\si_\infty$ in the limit.   

\subsection{Plan of the paper}

In Section \ref{ReviewSec} we review work of  Giusti-Salvatore-Sinha on the cohomology of symmetric groups.
 The unstable Adams spectral sequence is typically intractable, with a non-abelian Quillen homology defining its $E_2$.  But in Section \ref{WidthSSSec}  we apply a standard result of Bousfield \cite{Bousfield} in the  special case that a cohomology ring is free, as is the case here, equating the $E_2$ with $Ext$ in the category of unstable modules over the Steenrod algebra of the desuspended indcomposables.  Using the skyline diagram presentation of the cohomology of the infinite symmetric group, these indecomposables are manageable.  

Indeed, we show that a filtration by skyline diagram width (which corresponds to composition length in the Kudo-Araki-Dyer-Lashof algebra) yields subquotients which are given by the Dickson algebras, modulo perfect squares.  The resulting width spectral sequence is relatively tractable, allowing us for example reveal an error, likely of transcription, in Wellington's Ext charts (at the 11- and 12- stems).   We then share computer calculations,
which imply  many more differentials than in the classical Adams spectral sequence, but regular phenomena as well.

  In particular, there are Bockstein ($h_0$) towers, which we classify in the width spectral sequence in Section \ref{TowersSec}.  These occur in the same dimensions as Wellington identified, though with considerably more effort, in the $E_2$ of the CWSS.  Thus, there are no differentials in the width spectral sequence with $h_0$ inverted, and we conjecture no differentials in the 
width spectral sequence in general, a purely algebraic question.   
It would be interesting to understand differentials between these $h_0$-towers and the special roles the resulting classes
in homotopy might play.

Based on a remarkable identification of the unstable Adams $E_2$ for $BO$ and some preliminary calculations, in Section \ref{JSec} we initiate the study of the $J$-homomorphism.  
We conjecture that the $J$ map from $SO$ to $Q_1 S^0$ induces
the map on the $E_2$ of the CWSS defined purely algebraically by reducing to the width one quotient in the width filtration.  Our  preliminary calculations, 
show that while the $J$ map does not induce this reduction on cohomology, the induced map on $Ext$ nonetheless agrees with the map induced by reduction.
If this conjecture holds, we could  imagine the width filtration being connected to the chromatic filtration more broadly, as 
 some basic numerology suggests as well.
 
 \subsection{Thanks} I would like to thank my advisor Dev Sinha for all of his help and guidance, Hal Sadofsky for many helpful conversations, Dan Isaksen for confirming interest in the aspects of the Curtis-Wellington spectral sequence which make up the bulk of this work, Haynes Miller for providing references, and Peter May for clarifying some of the history.

\section{Review}\label{ReviewSec}
	Recall that $H^*(Q_0S^0, \F_2) \cong H^*(BS_\infty, \F_2)$. Recent calculations of cohomology of finite symmetric groups by taking all symmetric groups together and considering both cup product and a transfer or induction product gives a concise presentation.
	
\begin{theorem}[GSS]\label{GSSHopfRing}
	 As a Hopf ring, $\bigoplus_n H^{*}( B\si_{n}; \F_{2})$ is generated by classes \\* $\gamma_{\ell [n]}$ $\in H^{n(2^{\ell} - 1)}(B \si_{n 2^{\ell}})$, along with unit classes on each component. The coproduct of $\gamma_{\ell [n]}$ is given by 
	 $$\Delta \gamma_{\ell [n]} = \sum_{i+j = n} {\gamma_{\ell[i]}} \otimes {\gamma_{\ell[j]}}.$$
	Relations between transfer products of these generators are given by 
	$$\gamma_{\ell [n] } \tr \gamma_{\ell [m] } = \binom{n+m}{n} \gamma_{\ell [n+m]}.$$
	Relations between cup products of generators are that cup products of generators
on different components are zero.
\end{theorem}
	
	These can be presented graphically, as ``skyline diagrams.''  The generators $\gamma_{\ell[n]}$ are represented by a rectangle of width $n\cdot 2^{\ell -1}$ and total area $n \cdot (2^{\ell}-1)$, so that the area of the rectangle corresponds to its degree in cohomology.  Cup product is indicated by vertical stacking to make columns, whose placement next to each other denotes transfer product.  We also draw in vertical dashed lines separating the block into $n$ equal sections, for purposes of coproduct. An example of such a diagram can be seen in Figure \ref{F:skylineEx}.
	
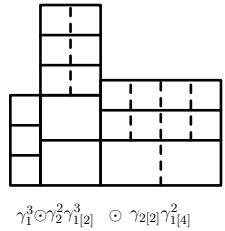
\begin{figure}[h]
\begin{center}
  \begin{tikzpicture}[line cap=round,line join=round,x=1.0cm,y=1.0cm, scale=0.4]

\draw [line width=1pt, color=black] (0,0) -- (1,0) -- (1,3) -- (0,3) -- cycle;
\draw [line width=1pt, color=black] (0,1) -- (1,1);
\draw [line width=1pt, color=black] (0,2) -- (1,2);

\node[align=center, scale=0.6] at (0.5,-1) {$\gamma_{1}^3 $};
\node[align=center, scale=0.6] at (1,-1) {$\odot$};

\draw [line width=1pt, color=black] (1,0) -- (3,0) -- (3,3) -- (1,3) -- cycle;
\draw [line width=1pt, color=black] (1,1.5) -- (3,1.5);
\draw [line width=1pt, color=black] (1,3) -- (3,3) -- (3,6) -- (1,6) -- cycle;
\draw [line width=1pt, color=black] (1,4) -- (3,4);
\draw [line width=1pt, color=black] (1,5) -- (3,5);
\draw [line width=1pt, dash pattern = on 3pt off 3pt, color=black] (2,3) -- (2,6);

\node[align=center, scale=0.6] at (2,-1) {$\gamma_{2}^2\gamma_{1[2]}^3$};
\node[align=center, scale=0.6] at (3.5,-1) {$\odot$};

\draw [line width=1pt, color=black] (3,0) -- (7,0) -- (7,3.5) -- (3,3.5) -- cycle;
\draw [line width=1pt, color=black] (3,1.5) -- (7,1.5);
\draw [line width=1pt, color=black] (3,2.5) -- (7,2.5);
\draw [line width=1pt, dash pattern = on 3pt off 3pt, color=black] (4,1.5) -- (4,3.5);
\draw [line width=1pt, dash pattern = on 3pt off 3pt, color=black] (5,0) -- (5,3.5);
\draw [line width=1pt, dash pattern = on 3pt off 3pt, color=black] (6,1.5) -- (6,3.5);

\node[align=center, scale=0.6] at (5,-1) {$ \gamma_{2[2]}\gamma_{1[4]}^2$};

\end{tikzpicture}
\caption{Skyline diagram for $\g_1^3 \odot \g_2\g_{1[2]}^3 \odot \g_{2[2]}\g_{1[4]}^2$}
\label{F:skylineEx}
\end{center}
\end{figure}

	The cohomology of the infinite symmetric group is the inverse limit
	\[ H^*(BS_\infty) = \varprojlim_n H^*(BS_n). \]
	The maps $H^*(BS_n) \to H^*(BS_m)$ for $n > m$ take a diagram in $H^*(BS_n)$  which has a ``tail''  of width   greater than or equal to $\frac{n-m}{2}$  (that is, a $\odot$-product factor of $1_k$ with $k > \frac{n-m}{2}$) to a diagram in $H^*(BS_m)$ obtained by shortening its tail to make it the appropriate width to be an element of $H^*(BS_m)$.  If a class is not such a transfer product with a sufficiently large unit class, it maps to zero.

	With restriction maps taking this form, the cohomology of $B\si_{\infty}$ could be viewed through such diagrams with ``infinitely long tails'', or in monomial form as  ``$\odot 1_\infty$''.   As they confer no additional information, we prefer to omit the tails altogether, only using them implicitly when we calculate through multiplication rules for finite groups. Using this presentation, we next recall another basic result of Giusti-Salvatore-Sinha, refining a classical result of Nakaoka.
 
	\begin{theorem}\label{cohomBSinfty}
	 	The cohomology of $BS_\infty$ is a polynomial algebra. 
Minimal generators of $H^*(BS_\infty)$ as an algebra under cup-product are represented graphically
by single columns with at least one block type appearing an odd number of times.
	\end{theorem} 
	
	These minimal generators form a basis for the indecomposibles of the cohomology of $BS_\infty$, which we call the Nakaoka module  $\mathfrak{N}$.

The idea of proof is that any product of such single-column diagrams will result in a sum of diagrams, the widest of which is the original skyline diagram. 
This shows such products are algebraically independent, and a simple filtration argument shows the polynomial ring they generate exhausts the cohomology.
For example, the skyline diagam from Figure \ref{F:skylineEx}, is the product of its columns, $\g_1^3$, $\g_{2}^2 \g_{1[2]}^3$, and $\g_{2[2]}\g_{1[4]}^2$ plus lower-width terms. 
	
	In the same paper, Giusti-Salvatore-Sinha describe the Steenrod algebra action on the cohomology of symmetric groups in terms of the basis elements $\g_{\ell[2^k]}$.  Because there are Cartan formulae for both cup and transfer product, this determines Steenrod structure on the whole.
	
	\begin{definition} (i) The algebraic degree of a gathered block is the total number of Hopf ring generators cup-multiplied to make the gathered block. \\
 (ii) The height of one of these skyline diagrams is the largest of the algebraic degrees of its constituent gathered blocks. \\
 (iii) The effective scale of a gathered block, composed of $\gamma_{\ell, n}$ cup-multiplied together, is the largest such $\ell$ that occurs in the block. The effective scale of a gathered monomial is the minimum of the effective scales of its gathered blocks. \\
 (iv) A monomial is full width as long as it is not a non-trivial transfer product of some monomial with some unit class $1_k$. 
	 \end{definition}

	\begin{theorem}[Theorem 8.3 of \cite{GSS}]\label{SteenrodAction}
	The Steenrod square $\Sq^i \gamma_{\ell[2^k]}$ is the sum of all full-width monomials of total degree $2^k(2^\ell -1 ) + i$, height one or two, and effective scale at least $\ell$, with height two only allowed if the effective scale $= \ell$. 
	\end{theorem}
	
	For example, Figure \ref{F:steenrod} illtustrates the three summands of $\Sq^3(\g_{2[4})$. 
	\[ \Sq^3(\gamma_{2[4]}) = \gamma_{4[1]} + \gamma_{3[1]} \tr \gamma_{2[1]}\gamma_{1[2]} \tr \gamma_{2[1]} + \gamma_{2[1]}^2 \tr \gamma_{2[1]} \tr \gamma_{2[2]} . \]

\begin{figure}[h]
\begin{center}
  \begin{tikzpicture}[line cap=round,line join=round,x=1.0cm,y=1.0cm, scale=0.35]

\node[align=center, scale=1] at (-13,0.75) {$\Sq^3$ (};
\draw [line width=1pt, color=black] (-11,0) -- (-3,0) -- (-3,1.5) -- (-11,1.5) -- cycle;
\draw [line width=1pt, dash pattern = on 3pt off 3pt, color=black] (-9,0) -- (-9,1.5);
\draw [line width=1pt, dash pattern = on 3pt off 3pt, color=black] (-7,0) -- (-7,1.5);
\draw [line width=1pt, dash pattern = on 3pt off 3pt, color=black] (-5,0) -- (-5,1.5);
\node[align=center, scale=1] at (-1.5,0.75) {) = };

\node[align=center, scale=0.7] at (-7,-1) {$\gamma_{2[4]}$};
    
\draw [line width=1pt, color=black] (0,0) -- (8,0) -- (8,1.875) -- (0,1.875) -- cycle;

\node[align=center, scale=0.7] at (4,-1) {$\gamma_{4[1]}$};
\node[align=center, scale=1] at (9,0.75) { + };

\draw [line width=1pt, color=black] (10,0) -- (14,0) -- (14,1.75) -- (10,1.75) -- cycle;
\draw [line width=1pt, color=black] (14,0) -- (16,0) -- (16,2.5) -- (14,2.5) -- cycle;
\draw [line width=1pt, color=black] (14,1.5) -- (16,1.5);
\draw [line width=1pt, dash pattern = on 3pt off 3pt, color=black] (15,1.5) -- (15,2.5);
\draw [line width=1pt, color=black] (16,0) -- (18,0) -- (18,1.5) -- (16,1.5) -- cycle;

\node[align=center, scale=0.7] at (14,-1) {$\gamma_{3[1]} \tr \gamma_{2[1]}\gamma_{1[2]} \tr \gamma_{2[1]}$};
\node[align=center, scale=1] at (19,0.75) { + };

\draw [line width=1pt, color=black] (20,0) -- (22,0) -- (22,3) -- (20,3) -- cycle;
\draw [line width=1pt, color=black] (20,1.5) -- (22,1.5);
\draw [line width=1pt, color=black] (22,0) -- (24,0) -- (24,1.5) -- (22,1.5) -- cycle;
\draw [line width=1pt, color=black] (24,0) -- (28,0) -- (28,1.5) -- (24,1.5) -- cycle;
\draw [line width=1pt, dash pattern = on 3pt off 3pt, color=black] (26,0) -- (26,1.5);

\node[align=center, scale=0.7] at (24,-1) {$ \gamma_{2[1]}^2 \tr \gamma_{2[1]} \tr \gamma_{2[2]}$};

\end{tikzpicture}
\caption{Skyline diagrams for the three summands of $\Sq^3(\gamma_{2[4]})$
}
\label{F:steenrod}
\end{center}
\end{figure}

It is straightforward to use Cartan formulae to calculate Steenrod action on indecomposibles $\mathfrak{N}$.
This gives a refinement of Nakaoka's work, which only determined this module additively, and is a much
more accessible presentation than of the homology primitives \cite{CLM,Madsen}.  Indeed, much
of Wellington's work on the CWSS is devoted to calculations with these primitives which are simplified or become
immaterial through this cohomology approach.

\section{Width spectral sequence}\label{WidthSSSec}
	In this section, we  equate the Curtis-Wellington $E_2$ with an explicit $\Ext$ group in the category of unstable modules over the Steenrod algebra, at which point there is an immediate filtration to develop. Recall that work of \cite{6Author}  and others proves that for simply connected $X$ with $\pi_*(X)$ of finite type, there is an unstable analog to the Adams spectral sequence with $E_2 \cong \Ext^{s,t}_{\ua}(H^*(X), \F_2)$ converging to $\pi_*(X)$. 
	
	 There have been relatively few computations made of the unstable Adams spectral sequence, with some explicit calculations of Curtis and Mahowald and Miller's proof of the Sullivan Conjecture being spectacular exceptions. A main roadblock is that the $\Ext$ groups which occur, which we call $\Ext_{\ua}$ for the category of unstable algebras over the Steenrod algebra, are not $\Ext$ groups in the usual sense of derived homomorphisms in an abelian category. While Goerss established that they are a ``non-abelian" derived $\Hom$, in the sense of Quillen, in the category of simplicial algebras over $\A$, this has not to our knowledge been used in any way for calculations. 
	
	To make calculations, one hopes for equivalent $\Ext$ calculations in abelian categories. The category of unstable modules over the Steenrod algebra, $\cu$, is abelian and there is a free unstable algebra functor $\cu \to \ua$. However, the cohomology of a space is very rarely in the image of this functor, even if it is free. If the cohomology is free only as an algebra, there is still an alternate form of reduction. Let $A$ be an augmented algebra with $\overline{A}$ its augmentation ideal. We let $\ind A$ denote the algebra indecomposables $\overline{A}/(\overline{A} \cdot \overline{A})$. Note that $\Sigma^{-1} \ind A$ is naturally in $\cu$. The following was originally stated by Bousfield in \cite{Bousfield} and follows from the composite functor spectral sequence constructed by Miller \cite{Miller}. 
	
	\begin{proposition}\label{abelianE2}
		Let $P$ be an unstable algebra over $\A$ that is free as an algebra on $\ind P$.Then 
		\[ \Ext^{s,t}_{\ua}(P,k) \cong \Ext^{s,t-1}_{\cu}(\Sigma^{-1} {\ind} P,k) \]
	\end{proposition}
	The unstable Asams spectral sequence, for $X = Q_0S^0$ calculates unstable homotopy groups, and we call this the Curtis-Wellington spectral sequence for stable homotopy groups. Nakaoka calculated \cite{Nakaoka71, Nakaoka72} that $H^*(Q_0S^0, \F_2)$ is free on its indecomposables, $\ind H^*(Q_0S^0, \F_2)$, which we will refer to as the Nakaoka module, $\mathfrak{N}$. Thus, Proposition \ref{abelianE2} applies to give the following reduction to a calculation in  the abelian category $\cu$
	
	\begin{proposition}
	The $E_2$ term of the Curtis-Wellington spectral sequence is isomorphic to $\Ext^{s,t}_{\cu}(\Sigma^{-1}\mathfrak{N})$.
	\end{proposition}
	
	Our graphical skyline diagram presentation, in particular of the indecomposables as stated in Theorem \ref{cohomBSinfty}, points immediately to a filtration of the Nakaoka module.  
	\begin{definition}\label{D:widthFiltration}
	 Let $F_n$ be the submodule of $\mathfrak{N}$ of elements of width less than or equal to $2^{n-1}$.	\end{definition}
	 
	  This is a submodule as unstable Steenrod modules by the formula from Theorem \ref{SteenrodAction}. Moreover, $F_n/F_{n-1}$ will be spanned by single columns of width exactly $2^{n-1}$ with at least one block type appearing an odd number of times.

Let $V_n$ be the transitive elementary abelian 2-subgroup of $S_{2^n}$.  Because restriction to a subgroup
maps to invariants of an action by the normalizer of the subgroup (see \cite{AdemMilgram}) in this case
the restriction map is received by rings of Dickson invariants.  But Corollary 7.6 of \cite{GSS}, follows an
argument of Milgram to show that restriction of $\g_{\ell, 2^k}$ with $\ell + k = n$ to  $V_n$ is the Dickson class $d_{k,l}$.  Thus single-column diagrams go to corresponding products of Dickson classes, and these single column skyline diagrams are isomorphic to a quotient of $D_n$ as a module over the Steenrod algebra. In particular, we have the following.

\begin{proposition} The quotient $F_n/F_{n-1}$ is isomorphic to $D_n^o$, where $D_n^o$ is the quotient of $D_n$ by all perfect squares.
\end{proposition}
This filtration allows us to consider only full width terms in the image of the Steenrod action, substantially simplifying calculations. Assembling the long exact sequences in cohomology associated to the short exact sequences 
	\[ 0 \to F_{n-1} \to F_n \to  D_n^o \to 0 \]
	from the filtration described above produces a tri-graded spectral sequence converging to the $E_2$ term of the Curtis-Wellington spectral sequence for stable homotopy.
	
	\begin{theorem}\label{WidthSS}
	The spectral sequence associated to the width filtration has 
	\[ E_1^{s, t;n } = \Ext_{\cu}^{s,t}( D_n^o, \F_2) \]
	and $d_r: \Ext^{s,t}( D_n^o) \to \Ext^{s+1,t}( D_{n+r}^o).$ It  converges to 
$\Ext^{s,t}_{\cu}(\Sigma^{-1} \mathfrak{N})$, and thus the $E_2$ of the Curtis-Wellington spectral sequence.
	\end{theorem}

	Using the well-known Steenrod structure on Dickson algebras, as presented for example in \cite{Hung}, we have been able to make hand calculations out to the 17 stem. Hood Chatham \cite{Hood} kindly produced an $\Ext$-chart illustrating the first page of this spectral sequence, out to the the forty-five stem. We share a clip here in Figure~\ref{ExtChart} and the full chart in the Appendix. Along the zero-line in this width spectral sequence we see the Steenrod indecomposables of the Dickson algebras, which have been of interest, for example in work of Hung and Peterson \cite{HungPeterson}, and are far from understood. This zero-line receives the Hurewicz map for $Q_0S^0$, as studied by Lannes and Zarati \cite{LannesZarati}.
		
	Wellington made similar computations for $\Ext(H_* Q_0S^0)$ at the prime 2, including charts out to the 17 stem. Comparing Wellington's results to ours, they mainly agree, but our calculations reveal an error in the 11 and 12 stem in the $\Ext$ chart. The original chart had a class in bidegree (12,4) and a $d_2$ differential to the class in degree $(11,6)$. Instead, our hand calculations show that there is a class in degree $(12,3)$, which appears in the computer calculations in degree $(11,3)$ after the desuspension. Since we know that the CWSS must converge to stable homotopy, we know that there must be a $d_3$ differential instead of the $d_2$ differential. 
	\pagebreak
	
\begin{center}
	\begin{figure} [h]
	\includegraphics[scale = 0.5]{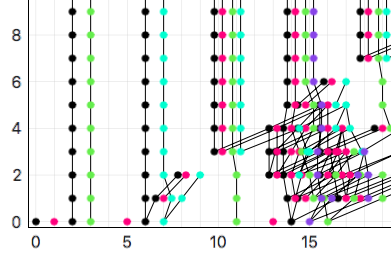}
	\caption{The $E_1$ page of the width spectral sequence, with width filtration encoded by color: 
black corresponds to $D_1$, red to $D_2$, green to $D_3$, teal to $D_4$, and purple to $D_5$.}
	\label{ExtChart}
	\end{figure}
	\end{center}	
	 
\DeclareSseqGroup\tower {} {
    \class(0,0)
    \foreach \y in {1,...,13} {
        \class(0,\y)
	\structline
	}
}

\begin{sseqdata}[name = Wellington results, Adams grading, xscale = 0.6, yscale = 0.4, y range = {0}{8}]
\class(1,0) 
\class(2,0) 
\tower(3,0) 
\tower(4,0) 

\foreach \y in {0,...,8} {
	\d[blue]3(4,\y)
}

\class(6,0) 
\tower(7,0)
\tower(8,0)
\foreach \y in {0,...,8} {
	\d[red]4(8,\y)
}
\class(8,1) 
\class(8,1) 
\class(9,1) 
\class(9,2) 
\class(9,2) 
\class(10,2) 
\tower(11,3) 

\tower(11,3) 

\tower(12,0) 
\foreach \y in {0,...,8} {
	\d[blue]3(12,\y)
}
\tower(12,3) 
\foreach \y in {3,...,8} {
	\d[blue]3(12,\y, 2, 2)
}
\class(14,0)
\class(14,2)
\class(14,2)
\class(14,3)
\class(14,3)
\class(14,4)
\class(14,4)

\tower(15,0)
\d[green]2(15,0,1,1)
\d[green]2(15,1,1,2)
\d[green]2(15,2,1,2)

\tower(15,1)
\d[green]2(15,1,2,1)
\d[green]2(15,2,2,1)
\class(15,2)
\class(15,3)
\class(15,3)  
\class(15,3) 
\class(15,4) 

\tower(16,0) 
\d[green]2(16,0,1,3)
\d[green]2(16,1,1,5)
\d[green]2(16,2,1,3)
\foreach \y in {3,...,8} {
	\d[purple]5(16,\y)
}

\tower(16,1) 
\foreach \y in {1,...,8} {
	\d[green]2(16,\y,2,2)
}
\class(16,1) 
\class(16,1) 
\d[green]2(16,1,3,3)

\class(16,2)
\class(16,3) 
\class(16,4)
\class(16,4)
\class(16,5)
\class(16,5)

\class(17,0)
\d[green]2(17,0,1,3)
\class(17,1)
\d[blue]3(17,1,1,4)
\class(17,1)
\d[green]2(17,1,2,3)
\class(17,2)
\class(17,2)
\d[green]2(17,2,2,3)
\class(17,2)
\d[blue]3(17,2,3,3)
\class(17,3)
\class(17,3)
\class(17,3)
\class(17,4)
\class(17,4)
\class(17,4)
\class(17,5)
\class(17,5)
\class(17,5)
\class(17,6)
\class(17,6)

\class(18,1)
\d[green]2(18,1,1,1)
\class(18,1)
\d[green]2(18,1,2,3)
\class(18,2)
\d[green]2(18,2,1,1)
\class(18,2)
\d[green]2(18,2,2,2)
\class(18,2)
\d[green]2(18,2,3,3)
\class(18,3)
\d[green]2(18,3,1,1)
\class(18,3)
\d[green]2(18,3,2,2)
\class(18,3)
\d[blue]3(18,3,3,1)
	
\end{sseqdata}
\printpage[name =  Wellington results, title style = { align = center, text width = 5.5in }, title = {Part of Table 13 by Wellington \cite{Wellington} (with correction in the 12 stem), depicting the $E_2$ page and differentials of the CWSS.} ] \\

\printpage[name =  Wellington results,  title style = { align = center, text width = 5.5in }, title = {$E_{\infty}$ page of Curtis-Wellington spectral sequence up through the 17 stem.}, page = 6]

	The differentials of the spectral sequence of Theorem \ref{WidthSS}, namely $$d_r: \Ext^{s,t}(D_n^o) \to \Ext^{s+1,t}(D_{n+r}^o),$$  fix topological degree ($t$), 
increase co-bar length by 1 ($s$), and increase filtration by $r$. In the charts, they will be moving one unit left, one unit up, and in our representation of the third (width) grading by color move between different colors so that if the source is in $D_{n}^o$, the target is in $D_{n+r}^o$. Then we can see by hand through at least the 16-stem that there are no possible differentials in the width filtration spectral sequence.

	\section{$h_0$ towers} \label{TowersSec}
	
	One of the immediate differences between the Adams spectral sequence and Curtis-Wellington spectral 
	sequence is the presence of $h_0$ towers.  We classify these in the width spectral sequence.
	
	We begin by defining a partial action of $\Ext_{\A}(\F_2, \F_2)$.
	In the Yoneda approach to $\Ext$, an element of $\Ext^{s,t}_{\A}(\F_2, \F_2)$ is an extension of length $s$ from $\F_2$ to $\Sigma^t \F_2$, namely 
\[ 0 \to \Sigma^t \F_2 \to E_1 \to \cdots \to E_s \to \F_2 \to 0, \]
where the $E_i$ are  $\A$ modules.  An element of $\Ext^{p, q}_{\cu}(\Sigma^{-1} D_n^o, \F_2)$ is an extension of length $p$ from $\Sigma^{-1} D_n^o$ to $\Sigma^q \F_2$ 
\[ 0 \to \Sigma^q \F_2 \to F_1 \to \cdots \to F_{p} \to \Sigma^{-1} D_n^o \to 0 \]
where the $F_i$ are unstable $\A$ modules and the maps $F_i \to F_{i+1}$ have 
excess less than or equal to the degree of $F_{i+1}$.

\begin{definition}
Define a partial action of $\Ext^{s,t}_{\A}(\F_2, \F_2)$ on $\Ext^{p, q}_{\cu}(\Sigma^{-1} D_n^o, \F_2)$, defined when   $t-s+1 \le q$,
 by suspending the stable extension $q$ times and concatenating it on the left with the unstable extension to give an extension which
 defines an element of $\Ext^{s+p, t+q}_{\cu}(\Sigma^{-1} D_n^o, \F_2)$.
\end{definition} 

\[ 0 \to \Sigma^{t+q} \F_2 \to \Sigma^{q}E_1 \to \cdots \to \Sigma^{q} E_s \to \Sigma^{q} \F_2 \to F_1 \to \cdots \to F_{p} \to \Sigma^{-1} D_n^o \to 0.\]

	We next recall the $\Lambda$-algebra \cite{6Author, BousfieldCurtis}, which gives an explicit though computationally involved way to compute some $\Ext$ groups over the Steenrod algebra. The $\Lambda$ algebra is the graded associative differential algebra with unit over $\F_2$ with 
	\begin{enumerate}
	\item a generator $\lambda_i$ of degree $i$ for each $i \ge 0$
	\item for each $i$, $k \ge 0$ a relation 
	\[ \lambda_i\lambda_{2i+1+k} = \Sigma_{j\ge 0} \binom{k-1-j}{j} \lambda_{i+k-j}\lambda_{2i+1+j} \]
	\item a differential $\partial$ given by
	\[ \partial(\lambda_i) = \Sigma_{j\ge 1} \Sigma_{j\ge 1} \binom{i-j}{j} \lambda_{i-j}\lambda_{j-1}. \]
	\end{enumerate}
	Note that $\Lambda = \bigoplus_{s\ge 0} \Lambda^s$ where $\Lambda^s$ is genreated by monomials $\lambda_I$ of length $s$. 
	
	We to define complexes using the $\Lambda$ algebra, we need our action to be on the right.
	Let $\mathcal{U_R}$ denote the category of unstable right $\A$ modules and continue to denote $\cu$ the category of unstable left $\A$ modules. For $M$ of finite type, $ \Ext^{s,t}_{\mathcal{U_R}}(\F_2, M) \cong \Ext^{s,t}_{\cu}(M^*, \F_2) .$

	The $\Lambda$ algebra gives one method to approach calculation of $\Ext$ groups, in particular, from \cite{BousfieldCurtis}, 
	\[ \Ext^{s,t}_{\mathcal{U_R}}(\F_2, M) \cong H^s(V(M))_{t-s} \]
	where $V(M)$ is the chain complex 
	\[ M \to M \hat\otimes \Lambda^1 \to M \hat\otimes \Lambda^2 \to \cdots  \]

	In \cite{Bousfield}, Bousfield defines the following tower complex as a quotient of the chain complex $V(M)$, as follows.
	\begin{definition} \label{def:towers} Let
	\[ T^s(M) = \begin{cases} 
		M \otimes (\lambda_0)^{s} & s = 0,1 \\
		M  \otimes (\lambda_0)^s \oplus \Sigma_{k>0} M_{2k} \otimes \lambda_{2k-1} (\lambda_0)^{s-1} & s > 1,\\ 
	\end{cases} \]
	with 
	\[ \delta(x \otimes \lambda_{2k-1} (\lambda_0)^{s-1}) = 0, \]
	and 
	\[ \delta(x \otimes (\lambda_0)^s) = \begin{cases}
		x \cdot \Sq^1 \otimes (\lambda_0)^{s+1} + x\cdot \Sq^{2k} \otimes \lambda_{2k-1}(\lambda_0)^s & s>0, x \in M_{4k} \\
		x \cdot \Sq^1 \otimes (\lambda_0)^{s+1} & \text{ otherwise}. \\ 
	\end{cases} \]
	\end{definition}

	Remark 2.4 of \cite{Bousfield} notes that the towers in $H^*(T(M))$ correspond with those in $H^*(V(M))$ and thus also with the towers in $\Ext^{s,t}_{U_R}(\F_2, M)$. Applying these tower detectors to $\Ext^{s,t}_{\cu}(\Sigma^{-1}D_n^o, \F_2)$, we conclude the following two theorems. 
	\begin{theorem} \label{D1 towers}
		There are infinite towers in $\Ext^{s,t}(\Sigma^{-1}D_1^o, \F_2)$ in degrees $4a - 2$ for $a$ a positive integer. 
	\end{theorem}

	\begin{theorem}\label{Dn towers}
		Let $n$ be an integer greater than or equal to 2. There are towers in \\* $\Ext^{s,t}(\Sigma^{-1}D_n^o, \F_2)$ corresponding to all integer solutions of
		\[ (2^{n-2}-2^{n-3})a_1 + \cdots + (2^{n-2}-1)a_{n-2} + (2^{n-1}-1)b_{n-1} = k \]
		where at least one of $a_1, \ldots, a_{n-2}$ are odd and $n \ge 3$. And also towers corresponding to all integer solutions of
		\[ (2^{n-1}-2^{n-2})b_1 + \cdots + (2^{n-1}-1)b_{n-1} + (2^{n}-1)c_{n} + (2^n-1) = k. \]
	\end{theorem} 
	
	 Wellington, in  \cite{Wellington} also used these tower detecting complexes, but of course with his homology approach. Recall that the homology of $Q_0S^0$ is free under the product induced by loop sum on classes $Q^I$ where $I$ is admissible. Wellington proves there are towers in dimensions $4k-1$ and $4k$ generated by $Q^I$ either in degree $4k$ with excess 0 and some odd index, or $Q^I$ in degree $4k-1$ with final index odd and all others even. Our calculations agree with Wellington's in that the towers in $\Ext^{s,t}_{\cu}(\Sigma^{-1}D_n^o, \F_2)$ for each $n$ correspond with his generated by $Q^I$ with $\ell(I) = n$.  Indeed, for $n = 4k-1$, we have towers corresponding to each integer solution to  
	\[ (2^{n-1}-2^{n-2})b_1 + \cdots + (2^{n-1}-1)b_{n-1} + (2^{n}-1)c_{n} + (2^n-1) = k. \]
	Each of these corresponds to the tower generated by $Q^I$ with $I = (s_1, s_2, \ldots, s_n)$ where $s_n = 2k + 1 + 2c_n$ and the other $s_i$ can be computed inductively from right to left with the formula
	\[ s_i = \frac{1}{2^{(n-1)-i}}\left(k+1+a_n+ \sum_{j = i}^{n-1} 2^{(n-1)-j}b_j \right) \text{ for } 2 \le i \le n-1, \]
	and finally
	\[ s_1 = 4k-1 -\sum_{j=2}^n s_j . \]
	For $n = 4k$, we found towers corresponding to each integer solution of 
	\[ (2^{n-2}-2^{n-3})a_1 + \cdots + (2^{n-2}-1)a_{n-2} + (2^{n-1}-1)b_{n-1} = k. \]
	Each of these corresponds to the tower generated by $Q^I$ with $I = (s_1, s_2, \ldots, s_n)$ where we once again find each term in the index working from right to left. First, $s_n = 2k$ and $s_{n-1} = k+ b_{n-1}$, then inductively we compute
	\[ s_i = \frac{1}{2^{(n-1)-i}}\left(k + b_{n-1} + \sum_{j=i}^{n-2} 2^{(n-2)-j}a_j \right) \text{ for } 2\le i \le n-1, \]
	and finally
	\[s_1 = 4k-\sum_{j=2}^n s_j. \]
	
	This agreement between our towers in the width spectral sequence and Wellington's imply there are no differentials  in the width filtration spectral sequence
	with $h_0$ inverted. Between the fact that these results indicate there are no differentials between towers and that some differentials can be eliminated by hand calculations in low degrees, we wonder whether there are any differentials in the width spectral sequence at all, a purely algebraic question.

	\begin{proof} [Proof of Theorem \ref{D1 towers}]
		Let $d_1$ be the generator of $D_1$. Then elements of $D_1^o$ are $d_1^{2i+1}$ in degree $2i$ after desuspension. Take $\{x_1, x_3, \cdots, x_{2i+1}, \cdots\}$ as a basis for the linear dual, $(\Sigma^{-1} D_1^o)^*$ where $x_{2i+1}$ is dual to $d_1^{2i+1}$, so each $x_{2i+1}$ is in degree $2i$ in the desuspension.   Consider $M = (\Sigma^{-1} D_1^o)^*$ as an unstable right $\A$ module by defining the linear map $x_i \cdot \Sq^k$ as 
		\[ (x_i \cdot \Sq^k)(y) = x_i(\Sq^ky), \text{ for } y \in D_1^o. \]
		
		We can use Bousfield's tower detector to determine where there are towers in $\Ext^{s,t}_{\mathcal{U_R}}(\F_2, M)$. As defined in as defined in Definition \ref{def:towers}, $T^s(M)$ is constructed so that in degree 2 and above the next degree is constructed from the previous by multiplying by $\lambda_0$ on the right and the differential from degree one onward is the same as the previous degree but with an extra factor of $\lambda_0$ on the right. Thus, it is sufficient to compute $H^2(T(M))$ to determine where the towers are. 
		
		We can see in Definition \ref{def:towers} that the differential $\delta$ only involves $\Sq^1$ and $\Sq^{2k}$ and thus in this case is determined by the fact that 
		\[ x_i \cdot \Sq^1 = 0 \text{ and } x_{4k+1} \cdot \Sq^{2k} = x_{2k+1}. \]

		Elements of $H^2(T(M))$ come in three forms. First, all $x_{4k+1} \otimes (\lm_0)^2$ are not cycles since $\delta(x_{4k+1} \otimes (\lm_0)^2) = x_{2k+1} \otimes \lm_{2k-1}(\lm_0)^2$. Second, all $x_{2k+1} \otimes \lm_{2k-1}\lm_0$ are cycles, but also boundaries hit by $x_{4k+1} \otimes \lm_0$. Finally, all $x_{4k-1}\otimes (\lm_0)^2$ are cycles and can not be boundaries since the image of $T^1(M)$ is only elements of the form $x_{2k+1} \otimes \lm_{2k-1} \lm_0$. Thus, we get a tower for $x_{4k-1} \otimes (\lm_0)^s$ in degree $4k-2$ for each positive integer $k$. 
	\end{proof}
	
	\begin{proof} [Proof of Theorem \ref{Dn towers}]
		We know that $D_n$ is generated by $n$ elements in degrees $(2^n-2^i)$ for $0 \le i \le n-1$. Let these generators be represented by $d_{2^n-2^i}$. Then an arbitrary basis element of $D_n^o$ is of the form $d_{2^n-2^{n-1}}^{a_1} \cdots d_{2^n-2^{0}}^{a_n}$ with at least one $a_i$ odd. Let $x_{a_1, \ldots, a_n} \in (D_n^{o})^*$ denote the linear dual of $d_{2^n-2^{n-1}}^{a_1} \cdots d_{2^n-2^{0}}^{a_n}$ in degree $\Sigma_{i = 1}^n a_i(2^n-2^{n-i})$. 

		Working now in $(\Sigma^{-1}D_n^o)^*$, $x_{a_1, \ldots, a_n} \in (D_n^{o})^*$ will now be in degree  $\Sigma_{i = 1}^n a_i(2^n-2^{n-i}) - 1$. Let $M = (\Sigma^{-1}D_n^o)^*$ and consider the tower detecting complex $T(M)$. As described n the proof of Theorem \ref{D1 towers} it is sufficient to calculate $H^2(T(M))$ to determine the location of the towers. 

		Since the differential $\delta$ as defined in Definition \ref{def:towers} only uses $\Sq^1$ and $\Sq^{2k}$,we need only understand the right action of $\Sq^1$ on an arbitrary $x_{a_1, \ldots, a_n}$ and the right action of $\Sq^{2k}$ on $x_{a_1, \ldots, a_n}$ with $\Sigma_{i = 1}^n a_i(2^n-2^{n-i}) = 4k+1$. These are given by the formulas
		\[ x_{a_1, \ldots, a_n} \cdot \Sq^1 = \begin{cases}
			x_{a_1, \ldots, a_{n-2}, a_{n-1}+1, a_{n} -1} & \text{ if $a_{n-1}$ even, $a_n \ge 1$, at least one $a_i$ odd} \\ 
			0& \text{ else}
		\end{cases} \]
		and if $\Sigma_{i = 1}^n a_i(2^n-2^{n-i}) = 4k+1$
		\[ x_{a_1, \ldots, a_n} \cdot \Sq^{2k} = \begin{cases}
			x_{\frac{a_1+2}{2}, \frac{a_2}{2}, \ldots, \frac{a_{n-1}}{2}, \frac{a_{n}-1}{2}} & \text{ if $a_{n}$ odd, $a_i$ even}\\ 
			x_{\frac{a_1}{2},\ldots, \frac{a_{j-1}}{2}, \frac{a_j-1}{2}, \frac{a_{j+1}+2}{2}, \frac{a_{j+2}}{2}  \ldots, \frac{a_{n-1}}{2}, \frac{a_{n}-1}{2}} & \text{ if $a_{j}, a_n$ odd, $a_i$ even for $i \ne j, n$}\\ 
			x_{\frac{a_1}{2},\ldots, \frac{a_{n-2}}{2},  \frac{a_{n-1}-1}{2}, \frac{a_{n}+1}{2}} & \text{ if $a_{n-1}, a_n$ odd, $a_i$ even for $1\le i < n-1$}\\ 
			0& \text{ else},
		\end{cases} \]
		where throughout $i$ and $j$ are between $1$ and $n-1$ inclusively.

		With these formulas in hand, we analyze $H^2(T(M))$. Our strategy will be to first characterize all cycles in $H^2(T(M))$ and then go through each type of cycle to determine which are boundaries. All of those that are not boundaries will correspond to our tower generators. 
		There are four types of cycles:
		\begin{enumerate}[label = (\alph*)]
		\item $x_{a_1, \ldots, a_n} \otimes (\lm_0)^2$ with $\Sigma_{i = 1}^n a_i(2^n-2^{n-i}) \ne 4k+1$ and $a_{n-1}$ odd 
		\item $x_{a_1, \ldots, a_n} \otimes (\lm_0)^2$ with $a_{n-1}$ even and $a_{n} = 0$ 
		\item $x_{a_1, \ldots, a_n} \otimes (\lm_0)^2$  with $\Sigma_{i = 1}^n a_i(2^n-2^{n-i}) = 4k+1$ and $a_n, a_{n-1}$, and some other $a_i$ for $1\le i < n-1$ odd
		\item $x_{a_1, \ldots, a_n} \otimes \lm_{2k-1}\lm_0$ with $\Sigma_{i = 1}^n a_i(2^n-2^{n-i}) = 2k+1$
		\end{enumerate}
		
		To see that these are all of the cycles, we look at our formulas for the action of $\Sq^1$ and $\Sq^{2k}$. As long as the degree of $x_{a_i, \ldots, a_n}$ is not a multiple of four, the differential on $x_{a_i, \ldots, a_n} \otimes (\lm_0)^2$ only involves $\Sq^1$. So, if $\Sigma_{i = 1}^n a_i(2^n-2^{n-i}) - 1 \ne 4k$, $x_{a_i, \ldots, a_n} \otimes (\lm_0)^2$ is a cycle when $x_{a_i, \ldots, a_n} \cdot \Sq^1 = 0$. This means either $a_{n-1}$ is odd and we get cycles of type (a), $a_n = 0$ and $a_{n-1}$ is even and we get cycles of type (b), or all $a_i$ are even, but then $x_{ai, \ldots, a_n}$ is not an element of $(D_1^o)^*$.
		
		Now if $x_{a_i, \ldots, a_n}$ is in degree $4k$, the differential on $x_{a_i, \ldots, a_n} \otimes (\lm_0)^2$ involves both $\Sq^1$ and $\Sq^{2k}$. Then $\Sigma_{i = 1}^n a_i(2^n-2^{n-i}) -1 = 4k$ which means that $a_n > 0$ and $a_{n-1}$ is even. For both $\Sq^1$ and $\Sq^{2k}$ to act trivially, $a_{n-1}$ must be odd and some other $a_i$ for $1\le i < n-1$ is also odd to give the cycles of type (c). 
		
		Finally, we have the elements of the form $x_{a_i, \ldots, a_n} \otimes \lambda_{2k-1}(\lambda_0)^{s-1}$ which are all cycles by definition and give us the cycles in class (d).

		Now, beginning with cycles of class (a), we want to determine which are also boundaries. For cycles in (a), $x_{a_1, \ldots, a_n} \otimes (\lm_0)^2$ with $\Sigma_{i = 1}^n a_i(2^n-2^{n-i}) \ne 4k+1$ and $a_{n-1}$ odd, we will split into three cases based on the value of $\Sigma_{i = 1}^n a_i(2^n-2^{n-i})$ mod 4. 
		
		 If $\Sigma_{i = 1}^n a_i(2^n-2^{n-i}) = 4k -1$ and $a_{n-1}$ is odd, then $a_n$ is odd and
\[ \delta(x_{a_1, \ldots, a_{n-2}, a_{n-1}-1, a_{n} + 1} \otimes \lm_0) = x_{a_1, \ldots, a_{n-2}, a_{n-1}, a_n} \otimes (\lm_0)^2 \]
as long as one of $a_1, \ldots, a_{n-2}$ are odd. Thus, we see that the only cycles not hit by boundaries are $x_{a_1, \ldots, a_{n-2}, a_{n-1}, a_n} \otimes (\lm_0)^2$ with $a_{n-1}, a_n$ odd and $a_i$ even for $1\le i \le n-2$.

		If $\Sigma_{i = 1}^n a_i(2^n-2^{n-i}) = 4k-2$ and $a_{n-1}$ is odd , then $a_n$ is even and 
		\[ \delta(x_{a_1, \ldots, a_{n-2}, a_{n-1}-1, a_{n} + 1} \otimes \lm_0) = x_{a_1, \ldots, a_{n-2}, a_{n-1}, a_n} \otimes (\lm_0)^2 \]
		and we see that all these cycles are also boundaries. 

		If $\Sigma_{i = 1}^n a_i(2^n-2^{n-i}) = 4k$ and $a_{n-1}$ is odd with $a_n \ge 2$, then 
$x_{a_1, \ldots, a_{n-2}, a_{n-1}, a_n} \otimes (\lm_0)^2  $ is the boundary of

		\[ \begin{cases} 
			\delta(x_{a_1, \ldots, a_{n-2}, a_{n-1}-1, a_{n} + 1} \otimes \lm_0 \\
\quad + x_{a_1+2, a_2, \ldots, a_{n-2}, a_{n-1}, a_{n}-1} \otimes \lm_0 ) & \text{ $a_{n-1}$ odd, all other $a_i$ even} \\
			\delta(x_{a_1, \ldots, a_{n-2}, a_{n-1}-1, a_{n} + 1} \otimes \lm_0 \\
\quad +x_{a_1, \ldots, a_{j-1}, a_{j}-1, a_{j+1} + 2, a_{j+2}, \ldots, a_{n-2}, a_{n-1}, a_n+1} \otimes \lm_0)  & \text{ $a_j$, $a_{n-1}$ odd, all other $a_i$ even} \\
			\delta(x_{a_1, \ldots, a_{n-2}, a_{n-1}-1, a_{n} + 1} \otimes \lm_0 ) & \text{  $a_{n-1}$ odd, at least 2 other $a_i$ odd} \\
		\end{cases}\]
		and we see that all these cycles are also boundaries. This now covers all cases for our class (a) cycles. 

		Turning to our class (b) cycles, $x_{a_1, \ldots, a_n} \otimes (\lm_0)^2$ with $a_{n-1}$ even and $a_{n} = 0$ we see that none of these are boundaries. Indeed, they would need to be in the image of some $x_{b_1, \ldots, b_n} \otimes (\lm_0)$ with $x_{b_1, \ldots, b_n}\cdot \Sq^1 = x_{a_1, \ldots, a_n}$. However, $\Sq^1$ changes the parity of each of the last two indices and is only nonzero if $b_{n-1}$ is even, but then its image $a_{n-1}$ must be odd, a contradiction. 

		Next, all class (c) cycles $x_{a_1, \ldots, a_n} \otimes (\lm_0)^2$  with $\Sigma_{i = 1}^n a_i(2^n-2^{n-i}) = 4k+1$ and $a_n, a_{n-1}$, and some other $a_i$ for $1\le i < n-1$ odd are boundaries hit by 
		\[ \delta(x_{a_1, \ldots, a_{n-2}, a_{n-1}-1, a_{n} + 1} \otimes \lm_0) = x_{a_1, \ldots, a_{n-2}, a_{n-1}, a_n} \otimes (\lm_0)^2. \] 
		Thus, all class (c) cycles are boundaries as well.

		Finally, for class (d), if $\Sigma_{i = 1}^n a_i(2^n-2^{n-i}) = 2k+1$
		\[ \delta(x_{2a_1, \ldots, 2a_{n-2}, 2a_{n-1}+1, 2a_{n}-1} \otimes \lm_0) = x_{a_1, \ldots, a_n} \otimes \lm_{2k-1}\lm_0. \]
		So these too are all boundaries. 

		This leaves us with only those classes of type (a) corresponding to $x_{a_1, \ldots, a_{n-2}, a_{n-1}, a_n} \otimes (\lm_0)^2$ with $a_{n-1}, a_n$ odd, $a_i$ even for $1\le i \le n-2$ and all cycles from class (b) with $x_{a_1, \ldots, a_n} \otimes (\lm_0)^2$ with $a_{n-1}$ even and $a_{n} = 0$ such that $\Sigma_{i = 1}^n a_i(2^n-2^{n-i}) = 4k$ in $H^2(T(M))$. These correspond to all integer solutions respectively of the equations:

		\[ (2^{n-1}-2^{n-2})b_1 + \cdots + (2^{n-1}-2)b_{n-2} +(2^{n-1}-1)b_{n-1} + (2^{n}-1)c_{n} + (2^n-1) = k. \]
		and 
		\[ (2^{n-2}-2^{n-3})a_1 + \cdots + (2^{n-2}-1)a_{n-2} + (2^{n-1}-1)b_{n-1} = k \]
		with at least one $a_i$ for $1 \le i < n-2$ odd.
	\end{proof}
	
	While these  calculations immediately give towers in the width filtration spectral sequence, we do not perform the analysis to determine which
	survive to the $E_2$ of the CWSS, instead citing agreement with Wellington's towers in the $E_2$ of the CWSS.  His argument is substantially more involved
	than what was required above, so we would like to find a self-contained argument at some point.
	
	It would be interesting to investigate the elements of homotopy which correspond to these towers, which must of course all support or receive differentials.  
	Would the resulting finite towers at $E_\infty$ be exceptional in any way?
	
\section{First analysis of J map} \label{JSec}
We start with a fun observation.  The Bousfield result \cite{Bousfield} equating the $E_2$ of the unstable Adams spectral sequence with an $\Ext$ in unstable
modules can be applied to $\Ext_{\ua}^{s,t}(H^*(BO), \F_2)$, as of course  $H^*(BO)$ is polynomial on the Stiefel-Whitney classes.  We get 
\[ \Ext^{s.t}_{\ua}(H^*(BO), \F_2) \cong \Ext^{s, t-1}_{\cu}(\Sigma^{-1} \ind H^*(BO), \F_2). \]

Up to decomposables, $\Sq^i(w_j) = {j-1 \choose i} w_{j+i}$, where $w_j$ is the $j$th Stiefel Whitney class \cite{Kochman, Stong, Wu}. Preliminary calculations for $\Ext_{\cu}$ lead to the observation that its $\Ext$ chart looked like $\Ext_{\cu}(\Sigma^{-1} D_1^o)$, but shifted to the right. This then lead us to the following isomorphism between familiar modules, which we found surprising. Moreover, this hints that the lowest degree in the width filtration yields the image of J, an idea that we expand on more later in this section. 

\DeclareSseqGroup\tower {} {
    \class(0,0)
    \foreach \y in {1,...,8} {
        \class(0,\y)
	\structline
	}
}

\begin{sseqdata}[name = BO, Adams grading, xscale = 0.6, yscale = 0.5, x range = {1}{17},y range = {0}{4}]
\class(1,0) 
\tower(3,0) 

\tower(7,0)

\class(8,1) 
\structline(7,0)(8,1)
\class(9,2) 
\structline(8,1)(9,2)

\tower(11,3) 

\class(14,2)

\class(14,3)
\structline(14,2)(14,3)
\class(14,4)
\structline(14,3)(14,4)
\structline(11,3)(14,4)

\tower(15,0)

\class(15,3)
\structline(14,2)(15,3,2)

\class(16,1)
\structline(15,0)(16,1)

\class(16,4)
\structline(15,3,2)(16,4)


\class(17,2)
\structline(16,1)(17,2)
\class(17,3)

\structline(14,2)(17,3)

\class(17,4)
\structline(14,3)(17,4)
\structline(17,3)(17,4)

\class(17,5)
\structline(14,4)(17,5)
\structline(16,4)(17,5)
\structline(17,4)(17,5)

\class(18,1)
\structline(15,0)(18,1)
\class(18,2)
\structline(15,1)(18,2)
\class(18,3)
\structline(15,2)(18,3)
\structline(17,2)(18,3)


\end{sseqdata}
\begin{figure} [h]
\printpage[name =  BO, title style = { align = center, text width = 5.5in }, title = {Calculations of $\Ext_{\cu}(BO,\F_2)$} ] \\
\end{figure}

\begin{proposition}\label{shiftedD1}
$\Ext_{\cu}^{s,t}(\Sigma^{-1}  D_1^o) \cong \Ext_{\cu}^{s,{t+1}}(\Sigma^{-1} \ind H^*(BO))$
\end{proposition} 

Before proving this proposition, we recall notation from \cite{Schwartz}. Define $\Phi: \cu \to \cu$ on $M \in \cu$ at the prime 2 to by 
\[ (\Phi M)^n \cong M^{n/2} \quad \Sq^i \Phi x = \Phi \Sq^{i/2} x \]
where $M^{n/2}$ is trivial if $n/2$ is not an integer. Let $\lambda_M: \Phi M \to M$ by $\Phi x \mapsto \Sq_0 x$. Define the functor $\Omega$ and its first (and only nontrivial) left derived function $\Omega_1$ by 
\[ \ker \lambda_M = \Sigma \Omega_1 M \quad \text{ and } \quad\coker\lambda_M = \Sigma \Omega M . \]

\begin{proposition}\label{injectiveLambdaExtIso}
If $M \in \cu$ and $\lambda_{M}$ is injective, then $\Ext_{\cu}^{s,t}(\Omega M, \F_2) \cong \Ext^{s,t+1}(M, \F_2)$. 
\end{proposition}

\begin{proof}
We know that 
\[ 0 \to \ker(\lambda_M) \to \Phi(M) \to M \to \coker(\lambda_M) \to 0 \] 
is an exact sequence where $\ker(\lambda_M) \cong \Sigma\Omega_1 M$ and $\coker(\lambda_M) \cong \Sigma\Omega M$. Since $\lambda_{M}$ is injective, $\ker(\lambda_M) = 0$, so $\Omega_i M = 0$ for all $i > 0$.  Let $P_{\bullet} \to M$ be a free resolution of $M$. Then $\Omega P_{\bullet} \to \Omega M$ is a free resolution of $\Omega M$. Thus, $\Ext_{\cu}(\Omega M , \F_2) \cong H_s(\Hom^t(\Omega P_{\bullet}, \F_2))$. Since $\Omega$ is left adjoint to $\Sigma$, we get
\[ H_s(\Hom^t(\Omega P_{\bullet}, \F_2)) \cong H_s(\Hom^t( P_{\bullet}, \Sigma \F_2)) \cong H_s(\Hom^{t+1}( P_{\bullet}, \F_2)) \cong \Ext^{s, t+1}_{\cu}(M, \F_2) \] 
\end{proof}

\begin{proof}(of Proposition \ref{shiftedD1}) 
The result will follow from Proposition \ref{injectiveLambdaExtIso} if we can show that $\lambda_M$ is injective for $M = \Sigma^{-1} \ind H^*(BO)$ and $\coker \lambda_M \cong \Sigma \Sigma^{-1} D_1^o$. 

For any $i > 0$, $w_i \in M$, $w_i$ has degree $i-1$ and 
\[ \lambda_M(w_i) = \Sq_0(w_i) = \Sq^{i-1}(w_i) = \binom{i-1}{i-1}w_{i+i-1} = w_{2i-1} \]
so $\lambda_M$ is clearly injective. We can also see that the image of $\lambda_M$ will be $w_{2i-1}$ for $i \ge 1$. Then $\coker \lambda_M$ will be $w_{2i} + \im(\lambda_M)$ in degrees $2i-1$. The map $w_{2i} + \im(\lambda_M) \mapsto d_1^{2i-1}$ then defines an isomorphism with $D_1^o \cong \Sigma\Sigma^{-1}D_1^o$, so we conclude that $\Sigma\Sigma^{-1}D_1^o \cong \coker \lambda_M$. Thus, applying Proposition \ref{injectiveLambdaExtIso}, $\Ext^{s,t}(\Sigma^{-1}  D_1^o) \cong \Ext^{s,{t+1}}(\Sigma^{-1} \ind H^*(BO))$.
\end{proof}

The suggestive calculation above for $BO$ could tell us about the delooping of the image of $J$, but our other calculations deal with $QS^0$ and not its delooping and are thus better suited to the image of $J$ map itself.  
For $QS^0$ itself, first calculations are consistent with the following.

\begin{conjecture}
The algebraic map on $\Ext$ induced by the image of $J$ on cohomology is the same as the map on $\Ext$ induced by reduction to $D_1^o$. 
\end{conjecture}

Evidence for this conjecture comes from hand calculations made out through degree 15, which we provide below.  
We will show that the image of $J$ on cohomology is itself not reduction to $D_1^o$, but it
seems to induce the same map on $\Ext$ as that reduction.  We further speculate that 
the induced map on spectral sequence $E_2$ terms is algebraic, given by the map induced on $\Ext$ induced by the map on cohomology.   Such a
phenomenon is rare (not occurring for example for maps between spheres) and we can present no evidence, only hope as that might then point to a new vantage
point from which to understand the chromatic filtration.

To make our calculations of this induced map on $\Ext$ we turn our attention to the basis for the homology of $SO$ presented by Hatcher \cite{Hatcher}, 
\[ H_*(SO, \Z) \cong E[e^1, \ldots, e^n, \ldots] \]
and the basis for $Q_1S^0$ in Milgram \cite{Milgram}, 
\[ H_*(Q_1S^0, \Z_2) \cong E(g_1, \ldots, g_i, \ldots) \otimes P(g_I) \]
where $I = (i_1, \ldots, i_m)$ and runs over all sequences of integers $0 \le i_1 \le \cdots \le i_m$ and $i_1 = 0$ implies $m = 2$ and $i_2 > 0$. 
Under the standard equivalence between different components of $Q S^0$, this corresponds to the Kudo-Araki-Dyer-Lashof basis by 
\begin{align*}
g_i \otimes 1 & \mapsto q_i(\iota)*\overline{\iota} \\
 1\otimes g_{0,i} &\mapsto q_i^2(\iota)*\overline{\iota}^3\\
 1 \otimes g_I &\mapsto q_I(\iota)*\overline{\iota}^{2^m-1}\\
 \end{align*}
  Then, the homology map induced by the $J$ homomorphism is $J_*: H_*(SO) \to H_*(Q_1S^0)$, defined by $ e_i \mapsto g_i \otimes 1 = q_i(\iota)*\overline{\iota}$. 

We can now use this and what we know of the pairings between homology and cohomology for both $SO$ and $Q_1S^0$ to see what we can understand about the induced map $J^*: \Sigma^{-1} \ind H*(Q_1S^0) \to \Sigma^{-1} \ind H^*(SO)$. We use the basis from Hatcher \cite{Hatcher}
\[ H^*(SO, \F_2) \cong \bigotimes_{i \text{ odd}} \F_2[\beta_i] \] 
where $\beta_i$ is the linear dual to $e^i$, and the skyline basis from \cite{GSS}. To work out the pairing between homology and cohomology for $SO$, we inductively use the Hopf algebra structure. We present pairing matrices in degree 11 and degree 15 to give an idea of what they look like in the Appendix. 

To work out the induced cohomology map on indecomposables, we go degree by degree through odd degrees only since indecomposables in $H^*(SO)$ are only in odd degrees.

In degree 1, the map is clear and $\g_{1[1]} \mapsto \beta_{1}$. \\

In degree 3, we have $e^3$ and $e^2 \wedge e^1$ in $H_*(SO)$. We then calculate 
\[ e^3 \mapsto q_3(\iota)*\bi \]
and
\[ e^2 \wedge e^1 \mapsto (q_{2}(\iota)*\bi) \circ (q_1(\iota)*\bi) = q_{1,1}(\iota)*\bi^3 + q_2(\iota)*q_1(\iota)*\bi^3 \] 
Thus, 
\[ J^*(q_3^*) = J^*(\g_{1[1]}^3) = (e^3)^* = \beta_3 \]
\[ J^*(q_{1,1}^*) = J^*(\g_{2[1]}) = (e^2\wedge e^1)^* = \beta_3 + \beta_1^3 \]
and 
\[ J^*((q_2*q_1)^*) = J^*(\g_{1[1]}^2\odot \g_{1[1]}) =(e^2\wedge e^1)^* = \beta_3 + \beta_1^3. \]

We only need to understand the map on indecomposables, which is just
\[ \g_{1[1]}^3 \mapsto \beta_3 \text{ and } \g_{2[1]} \mapsto \beta_3. \]  
As $\g_{2[1]}$ maps non-trivially on indecomposables, the map $J^*$ itself is not reduction to $D_1^o$.  But
We will see these give the same map on $\Ext$.

Indecomposables in $H^*(Q_1S^0)$ of odd degree only pair nontrivially with classes in $H_*(Q_1S^0)$ that are not nontrivial products. Thus, like we saw in this case with $q_2 *q_1$, any nontrivial $*$ products will be linear dual to only decomposables. From now on, we will ignore all products of 2 or more terms in the image of $J_*$. 

We performed calculations like these for each odd degree up to 15. Here we present only the degree 15 calculation as the others follow the same routine. In degree 15, homology classes can be products of up to five different $e^i$. However, using the $H_*(SO) / H^*(SO)$ pairing matrices, we will calculate that all products of three, four, or five $e^i$ are dual to only decomposables. The induced map on homology for products of one or two elements of total degree 15 is
\begin{align*}
 e^{15} &\mapsto q_{15} \\
e^{14} \wedge e^1 &\mapsto q_{1,7}+ \text{ products }\\
e^{13} \wedge e^2 &\mapsto q_{1,7}+ \text{ products }\\
e^{12} \wedge e^3 &\mapsto q_{1,7}+q_{3,6}+ \text{ products }\\
e^{11} \wedge e^4 &\mapsto q_{1,7}+ \text{ products }\\
e^{10} \wedge e^5 &\mapsto q_{1,7}+q_{3,6} + q_{5,5} + \text{ products }\\
e^{9} \wedge e^6 &\mapsto q_{1,7}+q_{5,5}+ \text{ products }\\
e^{8} \wedge e^7 &\mapsto q_{1,7} + \text{ products }
\end{align*}
From the $H_*/H^*$ pairing on $SO$, we have
\begin{align*}
(e^{15})^* &= \beta_{15} \\
(e^{14} \wedge e^1 )^* &= \beta_7^2\beta_1 + \beta_{15} \\
(e^{13} \wedge e^2 )^* &= \beta_{13}\beta_1^2 + \beta_{15} \\
(e^{12} \wedge e^3 )^* &= \beta_3^5 + \beta_{15} \\
(e^{11} \wedge e^4 )^* &= \beta_{11}\beta_1^4 + \beta_{15} \\
(e^{10} \wedge e^5 )^* &= \beta_5^3 + \beta_{15} \\
(e^{9} \wedge e^6 )^* &= \beta_9\beta_3^2 + \beta_{15} \\
(e^{8} \wedge e^7 )^* &= \beta_1^8\beta_7 + \beta_{15} 
\end{align*}
This combined with the homology maps give:
\begin{align*}
(q_{15})^* = \g_1^{15} &\mapsto \beta_{15} \\
(q_{1,7})^* = \g_{1[2]}^6\g_{2[1]}+\g_{2[1]}^5  &\mapsto (e^{14} \wedge e^1 )^*+(e^{13} \wedge e^2 )^*+(e^{12} \wedge e^3 )^*+(e^{11} \wedge e^4 )^*+(e^{10} \wedge e^5 )^* \\ &\quad\quad+(e^{9} \wedge e^6 )^* +(e^{8} \wedge e^7 )^*\\
& = \beta_7^2\beta_1+\beta_{13}\beta_1^2 + \beta_3^5 + \beta_{11}\beta_1^4 + \beta_5^3 + \beta_9 \beta_3^2 + \beta_1^8\beta_7 +\beta_{15} \\
(q_{3,6})^* = \g_{1[2]}^3\g_{2[1]}^3+\g_{2[1]}^5 &\mapsto (e^{12} \wedge e^3 )^* +  (e^{10} \wedge e^5 )^*  =\beta_3^5 +  \beta_5^3  \\
(q_{5,5})^* = \g_{2[1]}^5  &\mapsto (e^{10} \wedge e^5 )^* +(e^{9} \wedge e^6 )^* = \beta_5^3 +\beta_9\beta_3^2  
\end{align*}
So, up to decomposables: 
\[ \g_1^{15} \mapsto \beta_{15}, \quad \g_{1[2]}^6\g_{2[1]} \mapsto \beta_{15}, \quad \g_{1[2]}^3\g_{2[1]}^3 \mapsto 0, \quad \g_{2[1]}^5\mapsto 0, \quad \g_{1[4]}^2\g_{3[1]} \mapsto 0, \quad \g_{4[1]} \mapsto 0. \]

In general, it appears that the induced map on cohomology is given by the following conjecture. 
\begin{conjecture}\label{conj:cohomMap}
The map on cohomology indecomposables induced by the $J$ homomorphism is:
\[ \g_{1[1]}^{2k+1} \mapsto \beta_{2k+1} \]
\[ \g_{1[2]}^{i}\g_{2[1]} \mapsto \beta_{2i+3} \]
and all other indecomposables map to 0. 
\end{conjecture}

As mentioned above,  at the level of modules the induced map $J^*$  is not just reduction to $D_1^o$. However, it appears that we do see this at the level of resolutions. 

\begin{conjecture}
There exists a presentation of the Nakaoka module so that all generators other than $\g_{1[1]}^{2^k-1}$ go to 0 under the map induced by the $J$ map. 
\end{conjecture}

To see this in low degrees, choose generators $\g_{1[1]}$ in degree 1, $\g_{1[2]}$ in degree 2, $\g_{1[1]}^3$ in degree 3, and $\g_{1[4]}$ in degree 4. Then Figure \ref{F:moduleDiag} illustrates one choice of presentation for the first 5 degrees of the Nakaoka module as an unstable module over the Steenrod algebra. From the map we calculated, $\g_{1[1]}^{2k-1}$ all map to $\beta_{2k-1}$ while $\g_{1[2]}$ and $\g_{1[4]}$ map to 0. Then, we see that since both $\g_{1[1]}^3$ and $\g_{2[1]}$ map to $\beta_3$ and both $\g_{1[1]}^5$ and $\g_{1[2]}\g_{2[1]}$ map to $\beta_5$ that both those elements map to 0. 

\begin{figure}[h]
\begin{center}
\begin{tikzpicture}[line cap=round,line join=round,x=1.0cm,y=1.0cm, scale=0.8]

\filldraw (0,4) circle (1pt);
\node[align=center, scale=0.75] at (0.3,4) {$ \gamma_{1[1]}^5$};

\draw (0,4) ..controls (-0.5,3)..  (0,2);

\filldraw (0,2) circle (1pt);
\node[align=center, scale=0.75] at (0.3,2) {$ \gamma_{1[1]}^3$};

\draw (0,2) ..controls (-0.5,1)..  (0,0);

\filldraw (0,0) circle (1pt);
\node[align=center, scale=0.75] at (0.3,0) {$ \gamma_{1[1]}$};

\filldraw (2,4) circle (1pt);
\node[ align=center, scale=0.75] at (3.2,4) {$ \gamma_{1[1]}^5+\gamma_{1[2]}\g_{2[1]}$};

\draw (2,3) -- (2,4);

\filldraw (2,3) circle (1pt);
\node[align=center, scale=0.75] at (2.3,3) {$ \gamma_{1[4]}$};

\draw (2,2) ..controls (1.5,3)..  (2,4);

\filldraw (2,2) circle (1pt);
\node[align=center, scale=0.75] at (2.9,2) {$ \gamma_{1[1]}^3+\gamma_{2[1]}$};

\draw (2,1) -- (2,2);

\filldraw (2,1) circle (1pt);
\node[align=center, scale=0.75] at (2.3,1) {$ \gamma_{1[2]}$};

\end{tikzpicture}
\caption{Diagram of $\mathfrak{N}$ as an unstable $\A$ module. }
\label{F:moduleDiag}
\end{center}
\end{figure}
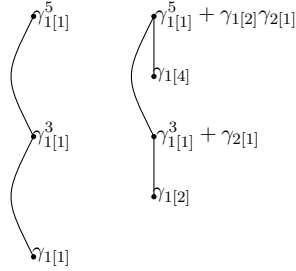

\newpage

\section{Appendix}\label{Appendix}

\subsection{The width spectral sequence.}  Through  degree 16, there are no possible differentials in the width spectral sequence so this  gives the $E_2$ of the CWSS
\begin{center}
\begin{figure}[bp]
\includegraphics[scale = 0.55]{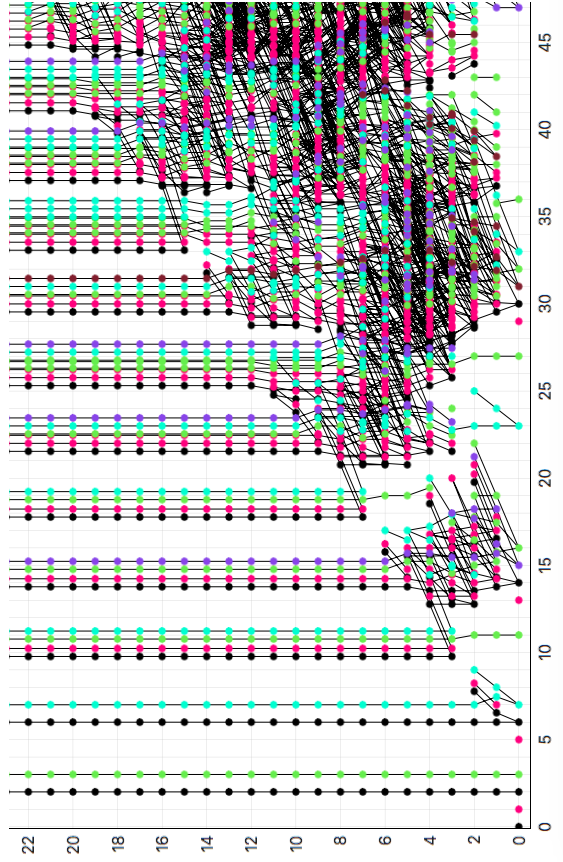}
\caption{$E_1$ page of the width spectral sequence. Black corresponds to $D_1$, red to $D_2$, green to $D_3$, teal to $D_4$, purple to $D_5$, and brown to $D_6$.}
\end{figure}
\label{F:FullChart}
\end{center}

 \subsection{Pairing Matrices for the homology and cohomology of $SO$} Here $e^{a_1,\ldots, a_n}$ corresponds to $e^{a_1} \wedge \cdots \wedge e^{a_n}$. We present only the matrices for degrees 11 and 15, both these and all others follow from the Hopf algebra structure. \\
 
Degree 11: \\

\begin{tabular}{c|c c c c c c c c c c c c }
& $e^{5 , 3 , 2 , 1}$ & $e^{8, 2 , 1}$ & $e^{7 , 3 , 1}$ & $e^{6 , 4 , 1}$ & $e^{6 , 3, 2}$ & $e^{5 , 4 , 2}$ &  $e^{10 , 1}$ & $e^{9, 2}$ & $e^{8 , 3}$ & $e^{7 , 4}$ & $e^{6 , 5}$ & $e^{11}$ \\
\hline
$\beta_5\beta_3\beta_1^3$ 	&1& 1 & 1 & 0 & 1 & 1 & 1 & 1 & 0 & 1 & 0 & 1 \\
$\beta_1^{11}$				& & 1& 0 & 0 & 0 & 0 & 1 & 1 & 1 & 0 & 0 & 1\\
$\beta_7\beta_3\beta_1$ 		&  &  & 1 & 0 & 0 & 0 & 1 & 0 & 1 & 1 & 0 & 1\\
$\beta_3^2\beta_1^5$ 		&  &  &  & 1 & 0 & 0  & 1 & 0 & 0 & 1 & 1 & 1 \\
$\beta_3^3\beta_1^2$ 		&  &  &  &  & 1 & 0 & 0 & 1 & 1 & 0 & 1 & 1 \\
$\beta_5\beta_1^6$ 			&  &  &  &  &  & 1 & 0 & 1 & 0 & 1 & 1 & 1 \\
$\beta_5^2 \beta_1$ 		&  &  &  &  &  &  & 1 & 0 & 0 & 0 & 0 & 1 \\
$\beta_9 \beta_1^2$ 		&  &  &  &  &  &  &  & 1 & 0 & 0 & 0 & 1 \\
$\beta_1^8 \beta_3$ 		&  &  &  &  &  &  &  &  & 1 & 0 & 0 & 1 \\
$\beta_7 \beta_1^4$ 		&  &  &  &  &  &  &  &  &  & 1 & 0 & 1 \\
$\beta_3^2 \beta_5$ 		&  &  &  &  &  &  &  &  &  &  & 1 & 1 \\
$\beta_{11}$ 				&  &  &  &  &  &  &  &  &  &  &  & 1 \\
\end{tabular}
\newpage

Degree 15: \\

In this degree the table is too large to fit on one page, so we break it into three separate charts. In the cases where an element of $H^*(SO)$ pairs trivially with all elements of $H_*(SO)$ in the table we do not include it in the matrix. \\

\begin{tabular}{c| c c c c c c c}
& $e^{5,4,3,2,1}$ & $e^{9,3,2,1}$ & $e^{8,4,2,1}$ & $e^{7,5,2,1}$ & $e^{7,4,3,1}$ & $e^{6,5,3,1}$ & $e^{6,4,3,2}$ \\
\hline
$\beta_1^7\beta_3\beta_5$ 	& 1&1&1&1&1&1&1\\
$\beta_1^3\beta_3\beta_9$ 	& &1&0&0&0&0&0\\
$\beta_1^{15}$ 				& &&1&0&0&0&0\\
$\beta_1^3\beta_5\beta_7$ 	& &&&1&0&0&0\\
$\beta_1^5\beta_3\beta_7$	& &&&&1&0&0\\
$\beta_1\beta_3^3\beta_5$ 	& &&&&&1&0\\
$\beta_1^6\beta_3^3$ 		& &&&&&&1\\
\end{tabular} \\

\begin{tabular}{c|c c c c c c c c c c c c}
& $e^{12 , 2 , 1}$ & $e^{11 , 3 , 1}$ & $e^{10 , 4 , 1}$ & $e^{10 , 3, 2}$ & $e^{9 , 5 , 1}$ & $e^{9 , 4 , 2}$ & $e^{8 , 6 , 1}$ & $e^{8 , 5 , 2}$& $e^{8,4,3}$& $e^{7,6,2}$& $e^{7,5,3}$ & $e^{6,5,4}$  \\
\hline
$\beta_1^7\beta_3\beta_5$ 	& 1&1&1&1&0&0&1&0&0&1&0&1\\
$\beta_1^3\beta_3\beta_9$ 	& 1&1&0&1&1&1&0&0&0&0&0&0\\
$\beta_1^{15}$ 				& 1&0&1&0&0&1&1&1&1&0&0&0\\
$\beta_1^3\beta_5\beta_7$ 	& 1&0&0&0&1&0&0&1&0&1&1&0\\
$\beta_1^5\beta_3\beta_7$	& 0&1&1&0&0&0&0&0&1&0&1&0\\
$\beta_1\beta_3^3\beta_5$ 	& 0&1&0&0&1&0&1&0&0&0&1&1\\
$\beta_1^6\beta_3^3$ 		& 0&0&0&1&0&1&0&0&1&1&0&1\\
$\beta_1^3\beta_3^4$ 		& 1&0&0&0&0&0&0&0&0&0&0&0\\
$\beta_1\beta_3\beta_{11}$ 	& &1&0&0&0&0&0&0&0&0&0&0\\
$\beta_1^5\beta_5^2$ 		& &&1&0&0&0&0&0&0&0&0&0\\
$\beta_1^2\beta_3\beta_{11}$ 	& &&&1&0&0&0&0&0&0&0&0\\
$\beta_1\beta_5\beta_9$ 		& &&&&1&0&0&0&0&0&0&0\\
$\beta_1^6\beta_9$ 			& &&&&&1&0&0&0&0&0&0\\
$\beta_1^{9}\beta_3^2$ 		& &&&&&&1&0&0&0&0&0\\
$\beta_1^{10}\beta_4$ 		& &&&&&&&1&0&0&0&0\\
$\beta_1^{12}\beta_3$ 		& &&&&&&&&1&0&0&0\\
$\beta_1^2\beta_3^2\beta_7$ 	& &&&&&&&&&1&0&0\\
$\beta_3\beta_5\beta_7$ 		& &&&&&&&&&&1&0\\
$\beta_1^4\beta_3^2\beta_5$ 	& &&&&&&&&&&&1\\
\end{tabular}

\begin{tabular}{c|c c c c c c c c}
& $e^{14 , 1}$ & $e^{13, 2}$ & $e^{12 , 3}$ & $e^{11 , 4}$ & $e^{10, 5}$ & $e^{9 , 6}$ &$e^{8,7}$ &$e^{15}$ \\
\hline
$\beta_1^7\beta_3\beta_5$ 	& 1&1&0&0&1&1&1&1\\
$\beta_1^3\beta_3\beta_9$ 	& 1&1&0&1&1&1&0&1\\
$\beta_1^{15}$ 				& 1&1&1&1&1&1&1&1\\
$\beta_1^3\beta_5\beta_7$ 	& 1&1&1&0&1&1&0&1\\
$\beta_1^5\beta_3\beta_7$	& 1&0&1&0&1&0&0&1\\
$\beta_1\beta_3^3\beta_5$ 	& 1&0&1&1&1&0&1&1\\
$\beta_1^6\beta_3^3$ 		& 0&1&1&1&1&0&1&1\\
$\beta_1^3\beta_3^4$ 		& 1&1&1&0&0&0&0&1\\
$\beta_1\beta_3\beta_{11}$ 	& 1&0&1&0&1&0&0&1\\
$\beta_1^5\beta_5^2$ 		& 1&0&0&1&1&0&0&1\\
$\beta_1^2\beta_3\beta_{11}$ 	& 0&1&1&0&1&0&0&1\\
$\beta_1\beta_5\beta_9$ 		& 1&0&0&0&1&1&0&1\\
$\beta_1^6\beta_9$ 			& 0&1&0&1&0&1&0&1\\
$\beta_1^{9}\beta_3^2$ 		& 1&0&0&0&0&1&1&1\\
$\beta_1^{10}\beta_4$ 		& 0&1&0&0&1&0&1&1\\
$\beta_1^{12}\beta_3$ 		& 0&0&1&1&0&0&1&1\\
$\beta_1^2\beta_3^2\beta_7$ 	& 0&1&0&0&0&1&1&1\\
$\beta_3\beta_5\beta_7$ 		& 0&0&1&0&1&0&1&1\\
$\beta_1^4\beta_3^2\beta_5$ 	& 0&0&0&1&1&1&0&1\\
$\beta_1\beta_7^2$ 			& 1&0&0&0&0&0&0&1\\
$\beta_1^2\beta_{13}$ 		& &1&0&0&0&0&0&1\\
$\beta_3^5$ 				& &&1&0&0&0&0&1\\
$\beta_1^4\beta_{11}$ 		& &&&1&0&0&0&1\\
$\beta_5^3$ 				& &&&&1&0&0&1\\
$\beta_3^2\beta_9$ 			& &&&&&1&0&1\\
$\beta_1^8\beta_7$ 			& &&&&&&1&1\\
$\beta_{15}$ 				& &&&&&&&1\\
\end{tabular}
\newpage

\renewcommand{\refname}{\normalfont\selectfont\normalsize References} 

\nocite{*}

\bibliographystyle{unsrtnat}
 \bibliography{CurtisSSBib}

\end{document}